\documentclass[12pt]{article}
\usepackage[pdftex]{color, graphicx}
\usepackage{setspace}
\usepackage[utf8]{inputenc}
\usepackage{amsmath}
\usepackage[pdftex,bookmarks,colorlinks]{hyperref}
\usepackage{latexsym}
\usepackage{amsfonts}
\usepackage{amssymb}
\usepackage{amsthm}
\usepackage{paralist}
\usepackage{mathrsfs}
\usepackage{enumitem}
\begin{document}
%\renewcommand{\labelitemi}{$\star$}
%\renewcommand{\labelitemi}{\textgreater}
%inparanenum itemizes with an inline list unlike enumerate%
\newtheorem{thm}{Theorem}
\newtheorem{cor}[thm]{Corollary}
\newtheorem{lem}{Lemma}
\newtheorem{prop}{Proposition}
\theoremstyle{remark}\newtheorem{rem}{Remark}
\theoremstyle{definition}\newtheorem{defn}{Definition}

%%%%%%%%%%%%%%%%%%%%%%%%%%%%%%%%%%%%%%%%%
\title{Variational Inequalities For The Differences Of Averages Over Lacunary Sequences}
%%%%
\author{Sakin Demir\\
Agri Ibrahim Cecen University\\ 
Faculty of Education\\
Department of Basic Education\\
04100 A\u{g}r{\i}, Turkey\\
E-mail: sakin.demir@gmail.com
}

%\date{June 05, 2021}

\maketitle

%% Classification and key words; note that the 2010 classification is used:

\renewcommand{\thefootnote}{}



\renewcommand{\thefootnote}{\arabic{footnote}}
\setcounter{footnote}{0}

%%%%
\begin{abstract} Let $f$ be a locally integrable function defined on $\mathbb{R}$, and let $(n_k)$ be a lacunary sequence. Define the operator $A_{n_k}$ by
$$A_{n_k}f(x)=\frac{1}{n_k}\int_0^{n_k}f(x-t)\, dt.$$
We prove various types of new inequalities for the variation operator
$$\mathcal{V}_sf(x)=\left(\sum_{k=1}^\infty|A_{n_k}f(x)-A_{n_{k-1}}f(x)|^s\right)^{1/s}$$
when  $2\leq s<\infty$.
\end{abstract}

%%%%%%%%%%%%%%%%%%%%%%%%%%%
An increasing sequence $(n_k)$ of real numbers is called lacunary if there exists a constant $\beta >1$ such that 
$$\frac{n_{k+1}}{n_k}\geq\beta$$
for all $k=0,1,2,\dots$.\\
\noindent
Let $f$ be a locally integrable function defined on $\mathbb{R}$. Let $(n_k)$ be a lacunary sequence and define the operator $A_{n_k}$ by
$$A_{n_k}f(x)=\frac{1}{n_k}\int_0^{n_k}f(x-t)\, dt.$$
It is clear that
$$A_{n_k}f(x)=\frac{1}{n_k}\chi_{(0, n_k)}\ast f(x)$$
where $\ast$ stands for convolution.\\
Consider the variation operator
$$\mathcal{V}_sf(x)=\left(\sum_{k=1}^\infty|A_{n_k}f(x)-A_{n_{k-1}}f(x)|^s\right)^{1/s}$$
for $2\leq s<\infty$.\\
Analyzing the boundlessness of the variation operator $\mathcal{V}_sf$ is a method of measuring the speed of convergence of the sequence $\{A_{n_k}f\}$.\\
%%%%%%%%%%%%%%%%%%%
 
Various types of inequalities for the two-sided variation operator
$$\mathcal{V}^{\prime}_sf(x)=\left(\sum_{-\infty}^\infty \left|\frac{1}{2^n}\int_x^{x+2^n}f(t)\, dt-\frac{1}{2^{n-1}}\int_x^{x+2^{n-1}}f(t)\, dt\right|^s\right)^{1/s}$$
when $2\leq s<\infty$ have been proven by the author in  S. Demir~\cite{sdemir}, in this research we prove that same types of inequalities are also true for any lacunary sequence $(n_k)$ for the one-sided variation operator $\mathcal{V}_sf(x)$ for $2\leq s<\infty$.

\begin{lem}\label{laclem}Let $(n_k)$ be a lacunary sequence with the lacunarity constant $\beta$, i.e.,
$$\frac{n_{k+1}}{n_k}\geq\beta >1$$
for all $k=0,1,2,\dots$, and $1\leq s <\infty$. Then there exists  a sequence $(m_j)$ such that
$$\beta^2\geq\frac{m_{j+1}}{m_j}\geq\beta >1$$
for all $j$ and 
$$\left(\sum_{k=1}^\infty|A_{n_k}f(x)-A_{n_{k-1}}f(x)|^s\right)^{1/s}\leq \left(\sum_{j=1}^\infty|A_{m_j}f(x)-A_{m_{j-1}}f(x)|^s\right)^{1/s}.$$
\end{lem}
\begin{proof}
Let us start our construction by first choosing $m_0=n_0$.  If
$$\beta^2\geq\frac{n_1}{n_0}\geq\beta,$$
define $m_1=n_1$.
If 
$$\frac{n_1}{n_0}>\beta^2,$$
let $m_1=\beta n_0$, then we have
$$\beta^2\geq\frac{m_1}{m_0}=\frac{\beta n_0}{n_0}=\beta\geq\beta .$$
Also, 
$$\frac{n_1}{m_1}\geq \frac{\beta^2n_0}{\beta n_0}=\beta .$$
Again, if 
$$\frac{n_1}{m_1}\leq\beta^2,$$
then choose $m_2=n_1$. If this is not the case, choose $m_2=\beta^2  n_0\leq n_1$.\\
By the same calculation as before, $m_0, m_1, m_2$ are part of a lacunary sequence satisfying
$$
\beta^2\geq\frac{m_{k+1}}{m_k}\geq\beta >1.
$$
To continue the sequence, either $m_3=n_1$ if
$$\frac{n_1}{m_2}\leq\beta^2$$
or $m_3=\beta^3n_0$ if 
$$\frac{n_1}{m_2}>\beta^2 .$$

Since $\beta>1$, this process will end at some $k_0$ such that $m_{k_0}=n_1$. The remaining elements $m_k$ are constructed in the same manner as the original $n_k$,  with necessary terms added between two consecutive $n_k$ to obtain the inequality
$$
\beta^2\geq\frac{m_{k+1}}{m_k}\geq\beta >1.
$$
Let now
$$J(k)=\{j:n_{k-1}<m_j\leq n_k\}$$
then we have
$$A_{n_k}f(x)-A_{n_{k-1}}f(x)=\sum_{j\in J(k)}(A_{m_j}f(x)-A_{m_{j-1}}f(x))$$
and thus we get
\begin{align*}
|A_{n_k}f(x)-A_{n_{k-1}}f(x)|&=\bigg|\sum_{j\in J(k)}(A_{m_j}f(x)-A_{m_{j-1}}f(x))\bigg|\\
&\leq \sum_{j\in J(k)}|A_{m_j}f(x)-A_{m_{j-1}}f(x)|.
\end{align*}
This implies that
\begin{align*}
\sum_{k=1}^\infty|A_{n_k}f(x)-A_{n_{k-1}}f(x)|&\leq \sum_{k=1}^\infty\sum_{j\in J(k)}|A_{m_j}f(x)-A_{m_{j-1}}f(x)|.\\
&= \sum_{j=1}^\infty |A_{m_j}f(x)-A_{m_{j-1}}f(x)|.
\end{align*}
Thus we have
$$\left(\sum_{k=1}^\infty|A_{n_k}f(x)-A_{n_{k-1}}f(x)|^s\right)^{1/s}\leq \left(\sum_{j=1}^\infty|A_{m_j}f(x)-A_{m_{j-1}}f(x)|^s\right)^{1/s}.$$
and this completes the proof.
\end{proof}
\begin{rem}\label{lacrem}We know from Lemma~\ref{laclem} that
$$\left(\sum_{k=1}^\infty|A_{n_k}f(x)-A_{n_{k-1}}f(x)|^s\right)^{1/s}\leq \left(\sum_{j=1}^\infty|A_{m_j}f(x)-A_{m_{j-1}}f(x)|^s\right)^{1/s}.$$
and the new sequence $(m_j)$ satisfies 
$$\beta^2\geq\frac{m_{j+1}}{m_j}\geq\beta >1$$
for all $j\in \mathbb{Z}^+$. Therefore,
we can assume without loss of generality,
$$\beta^2\geq\frac{n_{k+1}}{n_k}\geq\beta >1$$
for all $k\in \mathbb{Z}^+$ when we are proving any result for $\mathcal{V}_s(x)$.\\
\noindent
Since
$$\frac{1}{n_k}=\frac{n_1}{n_2}\cdot \frac{n_2}{n_3}\cdot\frac{n_3}{n_4}\cdot\cdots\cdot\frac{n_{k-1}}{n_k}$$
we can also assume that 
$$\frac{1}{n_k}\leq\frac{1}{\beta^{2(k-1)}}$$
for all $k=0,1,2,\dots$.
\end{rem}
%%%%%%%%%%%%%%%%%%%%%%%%%
\begin{lem}\label{lacseq}Let  $(n_k)$  be a lacunary sequence, and let $\gamma$ denote the smallest positive integer satisfying
$$\frac{1}{\beta}+\frac{1}{\beta^\gamma}\leq 1.$$
 If $i\geq j+\gamma$, $0<y\leq n_j$ and $n_j<x<n_{i+1}$, then
$$\chi_{(y, y+n_k)}(x)-\chi_{(0,n_k)}(x)=0$$
unless $k=i$ in which case
$$\chi_{(y, y+n_k)}(x)-\chi_{(0,n_k)}(x)=\chi_{(n_i,y+n_i)}.$$
\end{lem}
\begin{proof} Since $(n_k)$ is a lacunary sequence there exists a constant $\beta >1$ such that 
$$\frac{n_{k+1}}{n_k}\geq\beta$$ 
for all $k$.\\
We can assume that

\begin{equation}\label{eq1}
\beta^2\geq\frac{n_{k+1}}{n_k}\geq\beta
\end{equation}
for all $k$ by Remark~\ref{lacrem}.\\
\noindent
Since we have
$$ \frac{n_l}{n_k}=\frac{n_l}{n_{l+1}}\cdot\frac{n_{l+1}}{n_{l+2}}\cdot \dots \cdot\frac{n_{k-1}}{n_k}$$
and
$$\frac{1}{\beta}\leq\frac{n_k}{n_{k+1}}\leq \frac{1}{\beta^{k-l}}$$
for all $k$, we see that 
\begin{equation}\label{eq2}
\frac{1}{\beta^{2(k-l)}}\leq \frac{n_l}{n_k}\leq\frac{1}{\beta^{k-l}}
\end{equation}
for all $k>l$.\\
Let $\gamma$ denote the smallest positive integer satisfying
$$\frac{1}{\beta}+\frac{1}{\beta^\gamma}\leq 1.$$
We see from (\ref{eq2}) that
\begin{equation}\label{eq3}
n_j+n_k\leq n_{k+1}
\end{equation}
for all $k\geq j+\gamma -1$.\\
It is easy to see that for $k>i$,
$$0<y\leq n_j\leq n_i<x<n_{i+1}\leq n_k<y+n_k ,$$
and this implies that
$$\left[\chi_{(y,y+n_k)}(x)-\chi_{(0,n_k)}(x)\right]\cdot \chi_{(n_i,n_{i+1})}(x)=0.$$
For $k\leq i-1$, we see by (\ref{eq3}) that
$$n_k<y+n_k\leq n_j+n_{i-1}\leq n_i.$$
Then we have
$$\chi_{(y,y+n_k)}(x)\cdot \chi_{(n_i,n_{i+1})}(x)=\chi_{(0,n_k)}(x)\cdot \chi_{(n_i, n_{i+1})}=0.$$
Suppose now that $k=i$, by (\ref{eq3}) we have
$$y<n_i<y+n_i\leq n_j+n_i\leq n_{i+1}$$
and this implies that
$$\chi_{(y,y+n_i)}(x)-\chi_{(0,n_i)}(x)=\chi_{(y, y+n_i)}\cdot \chi_{(n_i, n_{i+1})}(x)=\chi_{(n_i, y+n_i)}(x).$$
\end{proof}
Let 
$$\phi_k(x)=\frac{1}{n_k}\chi_{(0, n_k)}(x)$$
and define the kernel operator $K:\mathbb{R}\to {\ell^s({\mathbb{Z}^+})}$ as 
$$K(x)=\{\phi_k(x)-\phi_{k-1}(x)\}_{k\in \mathbb{Z}^+}.$$
It is clear that
\begin{align*}
\mathcal{V}_sf(x)&=\|K\ast f(x)\|_{\ell^s({\mathbb{Z}^+})}\\
&=\left(\sum_{k=1}^\infty |\phi_k\ast f(x)-\phi_{k-1}\ast f(x)|^s\right)^{1/s}\\
&=\left(\sum_{k=1}^\infty|A_{n_k}f(x)-A_{n_{k-1}}f(x)|^s\right)^{1/s}
\end{align*}
where $\ast$ denotes convolution, i.e.,
$$K\ast f(x)=\int K(x-y)\cdot f(y)\, dy.$$
%%%%%%%%%%%%%%%
Let $B$ be a Banach space. We say that the $B$-valued  kernel $K$ satisfies $D_r$ condition, for $1\leq r<\infty$, and write $K\in D_r$, if there exists a sequence $\{c_l\}_{l=1}^\infty$ of positive numbers such that $\sum_lc_l<\infty$ and such that
$$\left(\int_{S_l(|y|)}\|K(x-y)-K(x)\|_B^r\, dx\right)^{1/r}\leq c_l|S_l(|y|)|^{-1/{r^\prime}},$$
for all $l\geq 1$ and all $y>0$, where $S_l(|y|)$ denotes the spherical shell $2^l|y|<|x|<2^{l+1}y$ and $\frac{1}{r}+\frac{1}{r^\prime}=1$.\\
When $K\in D_1$ we have the Hörmander condition:
$$\int_{|x|>2|y|}\|K(x-y)-K(x)\|_B\, dx\leq C$$
where $C$ is a positive constant which does not depend on $y>0$.\\
%%%%%%%%%%%%%%%%%%%%%%%%%
\begin{lem}\label{drlem} Let $\gamma$ denote the smallest positive integer satisfying
$$\frac{1}{\beta}+\frac{1}{\beta^\gamma}\leq 1.$$
and let $1\leq r,s<\infty$, $i\geq j+\gamma$, and $0<y\leq n_j$. Then 
$$\left(\int_{n_i}^{n_{i+1}}\|K(x-y)-K(x)\|^r_{\ell^s(\mathbb{Z}^+)}\, dx\right)^{1/r}\leq  C_in_i^{1/r-1},$$
i.e., $K$ satisfies $D_r$ condition for $1\leq r<\infty$.
\end{lem}
\begin{proof}Let
$$\Phi_k(x,y)=\phi_k(x-y)-\phi_k(x).$$
Then it is easy to check that
$$K(x-y)-K(x)=\{\Phi_k(x,y)-\Phi_{k-1}(x,y)\}_{k\in \mathbb{Z}^+}.$$
On the other hand, because of a property of the norm we have
\begin{align*}
\|K(x-y)-K(x)\|_{\ell^s(\mathbb{Z}^+)}&=\|\Phi_k(x,y)-\Phi_{k-1}(x,y) \|_{\ell^s(\mathbb{Z}^+)}\\
&\leq \|\Phi_k(x,y)\|_{\ell^s(\mathbb{Z}^+)}+\|\Phi_{k-1}(x,y) \|_{\ell^s(\mathbb{Z}^+)}\\
&\leq 2\|\Phi_{k-1}(x,y)\|_{\ell^s(\mathbb{Z}^+)},
\end{align*}
where $x$ and $y$ are fixed and $\|\Phi_{k-1}(x,y)\|_{\ell^s(\mathbb{Z}^+)}$ is the $\ell^s(\mathbb{Z}^+)$-norm of the sequence whose $k^{\text {th}}$-entry is $\Phi_k(x,y)$.\\
\noindent
We now have
\begin{align*}
\left(\int_{n_i}^{n_{i+1}}\|K(x-y)-K(x)\|^r_{\ell^s(\mathbb{Z}^+)}\, dx\right)^{1/r}&\leq 2\left(\int_{n_i}^{n_{i+1}}\|\Phi_{k-1}(x,y)\|^r_{\ell^s(\mathbb{Z}^+)}\, dx\right)^{1/r}\\
&\leq 2\left(\int_{n_i}^{n_{i+1}}\|\Phi_{k-1}(x,y)\|^r_{\ell^1(\mathbb{Z}^+)}\, dx\right)^{1/r}\\
&= 2\left(\int_{n_i}^{n_{i+1}}\left(\sum_{n_i<n_{k-1}} \frac{1}{n_{k-1}}\chi_{(n_i,y+n_i)}(x)\right)^{r}\, dx\right)^{1/r}\\
&= 2\left(\int_{n_i}^{n_{i+1}}\left(\sum_{n_i<n_{k-1}} \frac{1}{\beta^{2(k-2)}}\chi_{(n_i,y+n_i)}(x)\right)^{r}\, dx\right)^{1/r}\\
&\leq 2\left(\beta^{2}+\frac{1}{1-\beta^{2}}\right)\cdot \frac{1}{n_i}\cdot \left(\int_{n_i}^{n_{i+1}}\chi_{(n_i, y+n_i)}(x)\, dx\right)^{1/r}\\
&= 2\left(\beta^{2}+\frac{1}{1-\beta^{2}}\right)\cdot \frac{1}{n_i}\cdot y^{1/r}\\
&\leq  2 \left(\beta^{2}+\frac{1}{1-\beta^{2}}\right)\frac{1}{\beta^{(i-j)/r}}n_i^{1/r-1}
\end{align*}
where in the last inequality we used
$$y\leq n_j\leq\frac{n_i}{\beta^{i-j}}$$
by (\ref{eq2}), and this completes our proof with
$$C_i=2 \left(\beta^{2}+\frac{1}{1-\beta^{2}}\right)\frac{1}{\beta^{(i-j)/r}}.$$

\end{proof}

%%%%%%%%%%%%%%%%%%%%%%%%%%%%%%%%%%%%%%%%%%%%%%

\begin{lem}\label{ftrmbd}Let $\{n_k\}$ be a lacunary sequence then there exists a constant $C>0$ such that
$$\sum_{k=1}^\infty|\hat{\phi}_k(x)-\hat{\phi}_{k-1}(x)|<C$$
for all $x\in\mathbb{R}$, where $\phi_k(x)=\frac{1}{n_k}\chi_{(0,n_k)}(x)$, and $\hat{\phi}_k$ is its Fourier transform.
\end{lem}
\begin{proof}First note that we have
$$I(x)=\sum_{k=1}^\infty|\hat{\phi}_k(x)-\hat{\phi}_{k-1}(x)|=\sum_{k=1}^\infty\left|\frac{1-e^{-ixn_k}}{xn_k}-\frac{1-e^{-ixn_{k-1}}}{xn_{k-1}}\right|.$$
Let
$$I(x)=\sum_{\{k:|x|n_k\geq 1\}}|\hat{\phi}_k(x)-\hat{\phi}_{k-1}(x)|+\sum_{\{k:|x|n_k<1\}}|\hat{\phi}_k(x)-\hat{\phi}_{k-1}(x)|=I_1(x)+I_2(x).$$
Let us now fix $x\in\mathbb{R}$ and let $k_0$ be the first $k$ such that $|x|n_k\geq 1$. Since $\hat{\phi}_k(x)$ is an even function we can assume without the loss of generality that $x\geq 0$.\\
We clearly have
$$I_1(x)\leq\sum_{\{k:|x|n_k\geq 1\}}\frac{4}{|x|n_k}.$$
Since the sequence $\{n_k\}$ is lacunary there exists a constant $\beta >1$ such that
$$\frac{n_{k+1}}{n_k}\geq\beta$$
for all $k\in\mathbb{N}$. Also note that in summation,$I_1$, the term with index $n_{k_0}$ is the term with smallest index since it is the first term satisfiying condition $|x|n_k\geq 1$ and the sequence $\{n_k\}$ is increasing. On the other hand, we have
$$\frac{n_{k_0}}{n_k}=\frac{n_{k_0}}{n_{k_0+1}}\cdot\frac{n_{k_0+1}}{n_{k_0+2}}\cdot\frac{n_{k_0+2}}{n_{k_0+3}}\cdots \frac{n_{k-1}}{n_k}\leq\frac{1}{\beta^k}.$$
We now have
\begin{align*}
I_1(x)&\leq\sum_{\{k:|x|n_k\}}\frac{4}{|x|n_k}\\
&=\sum_{\{k:|x|n_k\geq 1\}}\frac{4n_{k_0}}{|x|n_{k_0}n_k}\\
&=\frac{4}{|x|n_{k_0}}\sum_{\{k:|x|n_k\geq 1\}}\frac{n_{k_0}}{n_k}\\
&\leq 4\sum_{\{k:|x|n_k\geq 1\}}\frac{1}{\beta^k}\\
\end{align*}
since $\frac{1}{|x|n_{k_0}}\leq 1$ and $\frac{n_{k_0}}{n_k}=\frac{1}{\beta^k}$. Also, since
$$\sum_{k=1}^\infty\frac{1}{\beta^k}=\frac{1}{1-\frac{1}{\beta}}$$
we clearly see that
$$I_1(x)\leq C_1$$
for some constant $C_1>0$.\\
%%%%%%%%%%%%%%%%%%%%%%%%%%%%%%%%%%%%%%%%%%%%%%%%
To control the summation $I_2$ let us first define the function $F$ as
$$F(r)=\frac{1-e^{-ir}}{r}, $$
then we have $\hat{\phi}_k(x)=F(xn_k)$. Now by the Mean Value Theorem there exists a constant $\xi \in (xn_k, xn_{k+1})$ such that
$$|F(xn_{k+1})-F(xn_k)|= |F^\prime (\xi )| |xn_{k+1}-xn_k|.$$
Also, it is easy to verify that
$$|F^\prime (x)|\leq \frac{x+2}{x^2}$$
for $x>0$.\\
\noindent
Now we have
\begin{align*}
|F(xn_{k+1})-F(xn_k)|&= |F^\prime (\xi )| |xn_{k+1}-xn_k|\\
&\leq \frac{\xi +2}{\xi^2}|x|(n_{k+1}-n_k)\\
&\leq \frac{xn_{k+1}+2}{x^2n_k^2}|x|(n_{k+1}-n_k)\\
&=\frac{2n_{k+1}}{n_k^2}(n_{k+1}-n_k).
\end{align*}
Thus we have
\begin{align*}
I_2(x)&=\sum_{\{k:|x|n_k<1\}}|F(xn_{k+1})-F(xn_k)|\\
&\leq \sum_{\{k:|x|n_k<1\}}\frac{2}{|x|n_k}\cdot \frac{2n_{k+1}}{n_k^2}(n_{k+1}-n_k)\\
&\leq   \sum_{\{k:|x|n_k<1\}}\frac{4n_{k+1}^2}{n_k^2|x|}\left(\frac{1}{n_k}-\frac{1}{n_{k+1}}\right)\\
&= \sum_{\{k:|x|n_k<1\}}\frac{16}{|x|}\left(\frac{1}{n_k}-\frac{1}{n_{k+1}}\right)\\
&=\frac{16}{|x|}\left(\frac{1}{n_1}-\frac{1}{n_{k_0+1}}\right)\\
&\leq \frac{16}{|x|n_{k_0+1}}\\
&\leq 16.
\end{align*}
We thus conclude that
$$I(x)=I_1(x)+I_2(x)\leq C_1+16=C$$
for all $x\in\mathbb{R}$ and this completes our proof.
\end{proof}
%%%%%%%%%%%%%%%%%%%%%%

\begin{lem}\label{strl2} Let $s\geq 2$ and $(n_k)$ be a lacunary sequence. Then there exits a constant $C>0$ such that
$$\|\mathcal{V}_sf\|_{ L^2(\mathbb{R})}\leq C\|f\|_{ L^2(\mathbb{R})}$$
for all $f\in L^2(\mathbb{R})$.
\end{lem}
\begin{proof}Since 
$$\sum_{k=1}^\infty|\hat{\phi}_k(x)-\hat{\phi}_{k-1}(x)|^2\leq \sum_{k=1}^\infty|\hat{\phi}_k(x)-\hat{\phi}_{k-1}(x)|,$$
it is clear from  Lemma~\ref{ftrmbd} that there exists a constant $C>0$ such that
$$\sum_{k=1}^\infty|\hat{\phi}_k(x)-\hat{\phi}_{k-1}(x)|^2<C$$
for all $x\in\mathbb{R}$.\\
\noindent
We now obtain
\begin{align*}
\|\mathcal{V}_sf\|_{ L^2(\mathbb{R})}&=\int_{\mathbb{R}}\left(\sum_{k=1}^\infty\left|\phi_k\ast f(x)-\phi_{k-1}\ast f(x)\right|^\rho\right)^{2/\rho}\, dx\\
&\leq \int_{\mathbb{R}}\sum_{k=1}^\infty\left|\phi_k\ast f(x)-\phi_{k-1}\ast f(x)\right|^2\, dx\\
&= \sum_{k=1}^\infty\int_{\mathbb{R}}\left|\phi_k\ast f(x)-\phi_{k-1}\ast f(x)\right|^2\, dx\\
&= \sum_{k=1}^\infty\int_{\mathbb{R}}\left|(\phi_k-\phi_{k-1})\ast f(x)\right|^2\, dx\\
&= \sum_{k=1}^\infty\int_{\mathbb{R}}\left|\Delta_k\ast f(x)\right|^2\, dx\;\;\;\;\;(\Delta_k(x)=\phi_k(x)-\phi_{k-1}(x))\\
&=\sum_{k=1}^\infty\int_{\mathbb{R}}|\widehat{\Delta_k\ast f}(x)|^2\, dx \;\;\;\;\;\textrm{(by Plancherel's theorem)}\\
&=\sum_{k=1}^\infty\int_{\mathbb{R}}|\widehat{\Delta_k}(x)|^2\cdot|\hat{ f}(x)|^2\, dx\\
&=\int_{\mathbb{R}}\sum_{k=1}^\infty|\widehat{\Delta_k}(x)|^2\cdot|\hat{ f}(x)|^2\, dx\\
&=\int_{\mathbb{R}}\sum_{k=1}^\infty|\hat{\phi}_k(x)-\hat{\phi}_{k-1}(x)|^2\cdot|\hat{ f}(x)|^2\, dx\\
&\leq C \int_{\mathbb{R}}|\hat{ f}(x)|^2\, dx\\
&=C\int_{\mathbb{R}}| f(x)|^2\, dx\;\;\;\;\;\textrm{(by Plancherel's theorem)}\\
&=C\|f\|_{ L^2(\mathbb{R})}^2
\end{align*}
as desired.
\end{proof}

%%%%%%%%%%

\begin{rem}\label{drrem} Since for $s \geq 2$, we have proved in Lemma~\ref{drlem} that the kernel operator $K(x)=\{\phi_k(x)-\phi_{k-1}(x)\}_{k\in\mathbb{Z}^+}$ satisfies $D_r$ condition for $1\leq r <\infty$, it specifically satisfies $D_1$ condition. We also have proved in Lemma~\ref{strl2} that $Tf=\{(\phi_k-\phi_{k-1})\ast f \}_{k\in\mathbb{Z}^+}$ is a bounded operator from $L^2(\mathbb{R})$ to $L^2_{\ell^s({\mathbb{Z}^+})}(\mathbb{R})$ since $\|K\ast f(x)\|_{\ell^s(\mathbb{Z}^+)}=\mathcal{V}_sf(x)$. Therefore,  $Tf=\{(\phi_k-\phi_{k-1})\ast f \}_{k\in\mathbb{Z}^+}$ is an $\ell^s$-valued singular operator of convolution type for $s \geq 2$.
\end{rem}

%%%%%%%%%%%%%%%%%%%%%%%%%%%%%%%%%%%%%%%%%%%%
\begin{lem}\label{singmaps} A singular integral operator $T$ mapping $A$-valued functions into $B$-valued functions can be extended to an operator defined in all $L_A^p$, $1\leq p<\infty$, and satisfying
\begin{enumerate}[label=\upshape(\roman*), leftmargin=*, widest=iii]
\item \label{sininta} $\|Tf\|_{L_B^p}\leq C_p\|f\|_{L_A^p},\quad 1<p<\infty,$
\item \label{sinintb} $\|Tf\|_{WL_B^1}\leq C_1\|f\|_{L_A^1},$
\item \label{sinintc} $\|Tf\|_{L_B^1}\leq C_2\|f\|_{H_A^1},$
\item \label{sinintd}$\|Tf\|_{{\rm{BMO}}(B)}\leq C_3\|f\|_{L^{\infty}(A)},\quad f\in L_c^{\infty}(A),$
\end{enumerate}
where $C_p,C_1,C_2,C_3>0$, and $L_c^{\infty}(A)$ is the space of bounded functions with compact support.
\end{lem}
\begin{proof}This is Theorem 1.3 of Part II in  J. L. Rubio de Francia \textit{et al}~\cite{jlrdffjrjlt}.
\end{proof}
The following theorem is our first result:
\begin{thm}\label{infvars}Let $2\leq s<\infty$, and $(n_k)$ be a lacunary sequence. Then there exits a constant $C>0$ such that
$$\|\mathcal{V}_sf\|_{ L^1(\mathbb{R})}\leq C\|f\|_{H^1(\mathbb{R})}$$
for all $f\in H^1(\mathbb{R})$.
\end{thm}
\begin{proof}This follows from Remark~\ref{drrem} and  Lemma~\ref{singmaps}\,(iii) since  $\|K\ast f(x)\|_{\ell^s(\mathbb{Z}^+)}=\mathcal{V}_sf(x)$.
\end{proof}
\begin{rem} We have proved that $Tf=\{(\phi_k-\phi_{k-1})\ast f \}_{k\in\mathbb{Z}^+}$ is an $\ell^s$-valued singular operator of convolution type for $s \geq 2$. By applying  Lemma~\ref{singmaps} to this observation we also provide  a different proof for the following known facts for $s \geq 2$ (see \cite{jkw1}) since  $\|K\ast f(x)\|_{\ell^s(\mathbb{Z}^+)}=\mathcal{V}_sf(x)$.
\begin{enumerate}[label=\upshape(\roman*), leftmargin=*, widest=iii]
\item  $\|\mathcal{V}_sf\|_{ L^p(\mathbb{R})}\leq C_p\|f\|_{ L^p(\mathbb{R})},\quad 1<p<\infty,$
\item $\|\mathcal{V}_sf\|_{WL^1(\mathbb{R})}\leq C_1\|f\|_{ L^1(\mathbb{R})},$
\item $\|\mathcal{V}_sf\|_{{\rm{BMO}(\mathbb{R})}}\leq C_2\|f\|_{L^{\infty}(\mathbb{R})},\quad f\in L_c^{\infty}(\mathbb{R}),$
\end{enumerate}
where $C_p,C_1,C_2>0$.
\end{rem} 
%%%%%%%%%%%%%%%%%%%%%
Let $w\in L_{\text{loc}}^1(\mathbb{R})$ be a positive function. We say that $w$ is an $A_p$ weight for some $1<p<\infty$ if the following condition is satisfied:
$$\sup_I\left(\frac{1}{|I|}\int_Iw(x)\, dx\right)\left(\frac{1}{|I|}\int_Iw(x)^{-\frac{1}{p-1}}\, dx\right)^{p-1}<\infty ,$$
where the supremum is taken over all intervals $I$ in $\mathbb{R}$.\\
We say that the function $w$ is an $A_\infty$ weight if there exist $\delta >0$ and $\epsilon >0$ such that given an interval $I$ in $\mathbb{R}$,  for any measurable $E\subset I$,
$$|E|<\delta\cdot |I|\implies w(E)<(1-\epsilon )\cdot w(I).$$
Here
$$w(E)=\int_Ew.$$
It is well known and easy to see that $w\in A_p\implies w\in A_\infty$ if $1<p<\infty$.\\
We say that $w\in A_1$ if  given an interval $I$ in $\mathbb{R}$ there is a positive constant $C$ such that
$$\frac{1}{|I|}\int_Iw(y)\, dy\leq Cw(x)$$
for a.e. $x\in I$.\\
%%%%%%%%%%%%%%%%
\begin{lem}\label{vecapfsing}Let $A$ and $B$ be Banach spaces, and $T$ be a singular integral operator mapping $A$-valued functions into $B$-valued functions with kernel $K\in D_r$, where $1<r<\infty$. Then, for all $1<\rho <\infty$, the weighted inequalities
$$\left\|\left(\sum_j\|Tf_j\|_B^{\rho}\right)^{1/\rho}\right\|_{L^p(w)}\leq C_{p,\rho}(w)\left\|\left(\sum_j\|f_j\|_A^{\rho}\right)^{1/{\rho}}\right\|_{L^p(w)}$$
hold if $w\in A_{p/r^\prime}$ and $r^\prime\leq p<\infty$, or if $w\in A_p^{r^\prime}$ and $1<p\leq r^\prime$.
Likewise, if $w(x)^{r^\prime}\in A_1$, then the weak type inequality

\begin{align*}
w\left(\left\{x:\left(\sum_j\|Tf_j(x)\|_B^{\rho}\right)^{1/\rho}>\lambda\right\}\right)
 \leq C_{\rho}(w)\frac{1}{\lambda}\int\left(\sum_j\|f_j(x)\|_A^{\rho}\right)^{1/\rho}w(x)\, dx
\end{align*}
holds.
\end{lem}
\begin{proof} This is Theorem 1.6 of Part II in J. L. Rubio de Francia \textit{et al}~\cite{jlrdffjrjlt}.
\end{proof}
Our next result is the following:
\begin{thm}\label{vecapfvar}Let $2\leq s<\infty$. Then, for all $1<\rho <\infty$, the weighted inequalities
$$\left\|\left(\sum_j(\mathcal{V}_sf_j)^{\rho}\right)^{1/\rho}\right\|_{L^p(w)}\leq C_{p,\rho}(w)\left\|\left(\sum_j|f_j|^{\rho}\right)^{1/{\rho}}\right\|_{L^p(w)}$$
hold if $w\in A_{p/r^\prime}$ and $r^\prime\leq p<\infty$, or if $w\in A_p^{r^\prime}$ and $1<p\leq r^\prime$.
Likewise, if $w(x)^{r^\prime}\in A_1$, then the weak type inequality
\begin{align*}
w\left(\left\{x:\left(\sum_j(\mathcal{V}_sf_j(x))^{\rho}\right)^{1/\rho}>\lambda\right\}\right)
 \leq C_{\rho}(w)\frac{1}{\lambda}\int\left(\sum_j|f_j(x)|^{\rho}\right)^{1/\rho}w(x)\, dx
\end{align*}
holds.
\end{thm}
\begin{proof}We have proved for $2\leq s<\infty$ that  $Tf=\{(\phi_k-\phi_{k-1})\ast f \}_{k\in\mathbb{Z}^+}$  is an $\ell^s$-valued singular integral operator of convolution type and its kernel operator $K(x)=\{\phi_k(x)-\phi_{k-1}(x)\}_{k\in\mathbb{Z}^+}$ satisfies $D_r$ condition for $1\leq r<\infty$. Thus the result follows from  Lemma~\ref{vecapfsing} and the fact that  $\|K\ast f(x)\|_{\ell^s(\mathbb{Z}^+)}=\mathcal{V}_sf(x)$.
\end{proof}
In particular we have the following corollary:
\begin{cor}Let $2\leq s<\infty$. Then the weighted inequalities
$$\left\|\mathcal{V}_sf\right\|_{L^p(w)}\leq C_{p,\rho}(w)\left\|f \right\|_{L^p(w)}$$
hold if $w\in A_{p/r^\prime}$ and $r^\prime\leq p<\infty$, or if $w\in A_p^{r^\prime}$ and $1<p\leq r^\prime$. Likewise, if $w(x)^{r^\prime}\in A_1$, then the weak type inequality
\begin{align*}
w\left(\left\{x:\mathcal{V}_sf(x)>\lambda\right\}\right)\leq C_{\rho}(w)\frac{1}{\lambda}\int |f(x)| w(x)\, dx
\end{align*}
holds.
\end{cor}

\end{document}